\documentclass[preprint,12pt]{elsarticle}

\usepackage{graphics}
\graphicspath{{.},{pics/}}

\usepackage{amssymb}
\usepackage{amsfonts}
\usepackage{mathptmx}
\usepackage{amsmath}
\usepackage{amsthm}
\usepackage{url}
\usepackage{color}

\newcommand{\defeq}{\mathrel{\mathop:}=}

\newcommand{\ud}{\mathrm{d}}
\newcommand{\R}{\mathbb{R}}

\newtheorem{theorem}{Theorem}

\newtheorem{lemma}[theorem]{Lemma}
\newtheorem{definition}[theorem]{Definition}

\newtheorem{remark}[theorem]{Remark}
\theoremstyle{definition}

\usepackage{algorithm}
 \biboptions{sort&compress}
\journal{Computational Statistics and Data Analysis}

\begin{document}

\begin{frontmatter}



\title{Adaptive Metropolis Algorithm Using Variational Bayesian Adaptive Kalman Filter}


\author[lut]{Isambi S. Mbalawata\corref{cor1}}%
\ead{Isambi.Mbalawata@lut.fi}
\author[aalto]{Simo S\"arkk\"a}
\author[jyv]{Matti Vihola}
\author[lut]{Heikki Haario}
\cortext[cor1]{Corresponding author} \address[lut]{Department of
  Mathematics and Physics, Lappeenranta University of Technology,
  P.O.Box 20, FI-53851 Lappeenranta, Finland}
\address[aalto]{Department of Biomedical Engineering and Computational
  Science, Aalto University, P.O.Box 12200, FI-00076 Aalto, Finland}
\address[jyv]{Department of Statistics, University of Oxford, 1 South Parks Road, Oxford, OX1 3TG, United Kingdom}

\begin{abstract}
  Markov chain Monte Carlo (MCMC) methods are powerful computational
  tools for analysis of complex statistical problems.  However, their
  computational efficiency is highly dependent on the chosen proposal
  distribution, which is generally difficult to find. One way to solve this problem is to use adaptive MCMC
  algorithms which automatically tune the statistics of a proposal
  distribution during the MCMC run.  
	A new adaptive MCMC algorithm, called the variational Bayesian adaptive Metropolis (VBAM) algorithm,
	is developed. The VBAM algorithm updates the proposal covariance matrix using the variational Bayesian adaptive Kalman filter (VB-AKF). A strong law of large numbers for the VBAM algorithm is proven. The empirical convergence results for three simulated examples and for two real data examples are also provided.
\end{abstract}

\end{frontmatter}

% \linenumbers

\section{Introduction}
Markov chain Monte Carlo (MCMC) methods \citep{Brooks+Gelman+Jones+Men:2011} are an important class of numerical
tools for approximating multidimensional integrals over complicated
probability distributions in Bayesian computations and in various
other fields. The computational efficiency of MCMC sampling depends on
the choice of the proposal distribution. A challenge of MCMC methods
is that in complicated high-dimensional models it is very hard to find
a good proposal distribution.

The Gaussian distribution is often used as a proposal distribution due
to its theoretical and computational properties. However, the Gaussian
proposal distribution needs a well tuned covariance matrix for optimal
acceptance rate and good mixing of the Markov chain. If the covariance matrix is too small, too large or has improper correlation structure, the Markov chains will be highly positively
correlated and hence the estimators will have a large variance. Because
manual tuning is laborious, several adaptive MCMC algorithms
have been suggested
\citep{Haario+Saksman+Tamminen:1999,Haario+Saksman+Tamminen:2001,Haario+Laine+Mira+Saksman:2006,Vihola:2012,
Andrieu+Thoms:2008,Roberts+Rosenthal:2007b,Roberts+Rosenthal:2009,Atchade+Rosenthal:2005,Gelman+Roberts+Gilks:1996}
to update the covariance matrix during the MCMC run.

In this article, we propose a new adaptive Metropolis algorithm, where
we update the covariance matrix of the Gaussian proposal distribution
of the Metropolis algorithm using the variational Bayesian adaptive Kalman
filter (VB-AKF) proposed by \citet{Sarkka+Nummenmaa:2009} and \citet{Sarkka+Hartikainen:2013}.
The idea of the classical Metropolis algorithm \citep{Haario+Saksman+Tamminen:1999} is essentially to empirically estimate the covariance of the samples and use this estimate to construct the proposal distribution. However, as we point out here, such a covariance estimation problem can also be formulated as an instance of recursive Bayesian estimation, where the term used for this kind of recursive estimation is \textit{Bayesian filtering} \citep{Sarkka:2013}. This reinterpretation allows one to construct alternative and potentially more effective adaptation mechanisms by utilizing the various Bayesian filtering algorithms developed over the years for doing the covariance estimation. The aim of this article is to propose a practical algorithm which is constructed from this underlying idea, prove its convergence, and test its performance empirically.

The structure of this article is the following: in Section
\ref{sec:mcmc} we review the existing adaptive MCMC methods. Section
\ref{sec:vbak} is dedicated to our new adaptive Metropolis algorithm.
Theoretical validity of the proposed algorithm is shown in Section
\ref{sec:lln} by proving a strong law of large numbers. In Section
\ref{sec:nume}, we study the empirical convergence of the method in three simulated examples and then apply the method
to two real data examples.

\section{Adaptive Markov Chain Monte Carlo Methods}\label{sec:mcmc}
Markov chain Monte Carlo (MCMC) methods are widely used algorithms for
drawing samples from complicated multidimensional probability
distributions. For example, in Bayesian analysis
\citep{Gelman+etal:2013}, we are often interested in
computing the posterior expectation of a function
$\mathbf{g}(\boldsymbol{\theta})$ given the measurements
$\mathbf{z}_1,\ldots,\mathbf{z}_M$:
\begin{equation}
  \mathrm{E}[\mathbf{g}(\theta) \mid \mathbf{z}_1,\ldots,\mathbf{z}_M]
  = \int_{\mathbb{R}^{d}} \mathbf{g}(\theta) \,
     p(\theta  \mid  \mathbf{z}_1,\ldots,\mathbf{z}_M) \, \mathrm{d}\theta.
\label{eq:moment}
\end{equation}
We can use MCMC methods to approximate the expectation by
drawing samples from the posterior distribution
\begin{equation}
  \boldsymbol{\theta}_1,\boldsymbol{\theta}_2,
    \ldots,\boldsymbol{\theta}_n
  \sim p(\boldsymbol{\theta}  \mid \mathbf{z}_1,\ldots,\mathbf{z}_M),
\end{equation}
and then by employing the approximation
\begin{equation}
  \mathrm{E}[\mathbf{g}(\boldsymbol{\theta}) \mid \mathbf{z}_1,\ldots,\mathbf{z}_M] \approx
  \frac{1}{n} \sum_{i=1}^n \mathbf{g}(\boldsymbol{\theta}_i).\label{eq:mcmc1}
\end{equation}
A common construction for MCMC uses a random walk that explores the state space
through local moves. The most well-known traditional MCMC method is
the \emph{Metropolis algorithm}. In the Metropolis algorithm we draw a candidate point
$\boldsymbol{\theta}_*$ from a symmetric proposal distribution
$q(\boldsymbol{\theta}_* \mid \boldsymbol{\theta})$ and use an
accept/reject rule to accept or reject the sampled point
\citep{Gilks+Richardson+Spiegelhalter:1996,Gelman+etal:2013,Brooks+Gelman+Jones+Men:2011}.

The efficiency of an MCMC algorithm can be improved by carefully tuning
the proposal distribution. Adaptive MCMC methods are a family of
algorithms, which take care of the tuning automatically. This proposal
is often chosen to be a Gaussian distribution, in which case it is the
covariance matrix that needs to be tuned. Under certain settings
\citet{Gelman+Roberts+Gilks:1996} show that the optimal covariance
matrix for an MCMC algorithm with  Gaussian proposal is $\lambda \,
\boldsymbol{\Sigma}$, with $\lambda=2.38^2/d$, where $d$ is the
dimension and $\boldsymbol{\Sigma}$ is the $d\times d$ covariance
matrix of the target distribution.

In the \emph{adaptive Metropolis (AM) algorithm} by
\citet{Haario+Saksman+Tamminen:2001}, the covariance matrix
$\boldsymbol{\Sigma}_{k-1}$ for the step $k$ is estimated as follows:
\begin{equation}\label{eq:covc}
  \boldsymbol{\Sigma}_{k-1} =
  \text{cov}(\boldsymbol{\theta}_0,
   \boldsymbol{\theta}_1,\ldots,\boldsymbol{\theta}_{k-1})
  +\epsilon \mathbf{I},
\end{equation}
where $\mathbf{I}$ is the $d\times d$ identity matrix and $\epsilon$ is
a small positive value whose role is to make sure that
$\boldsymbol{\Sigma}_{k-1}$ is not singular
\citep{Haario+Saksman+Tamminen:1999,Haario+Saksman+Tamminen:2001}.
The AM algorithm of \citet{Haario+Saksman+Tamminen:2001} can be
summarized as follows:
\begin{itemize}
\item Initialize $\boldsymbol{\theta}_0$, $\boldsymbol{\Sigma}_0$.
\item   For $k = 1,2,3,\ldots$
\begin{itemize}
\item Sample a candidate point $\boldsymbol{\theta}_*$ from a Gaussian distribution
\begin{equation}
  \boldsymbol{\theta}_* \sim \mathbb{N}(\boldsymbol{\theta}_{k-1}, \lambda \, \boldsymbol{\Sigma}_{k-1}).
\end{equation}

\item Compute the acceptance probability
\begin{equation}
 \alpha_k
  = \text{min}\left\{1,\, \frac{p(\boldsymbol{\theta}_* \mid \mathbf{z}_1,\ldots,\mathbf{z}_M)}{p(\boldsymbol{\theta}_{k-1} \mid \mathbf{z}_1,\ldots,\mathbf{z}_M)}\right\}.
\end{equation}

\item Sample a random variable $u$ from the uniform distribution $\mathbb{U}(0,1)$.

\item If $u<\alpha_k$, set $\boldsymbol{\theta}_{k} =  \boldsymbol{\theta}_*$.
Otherwise set $\boldsymbol{\theta}_{k} =  \boldsymbol{\theta}_{k-1}$.
\item Compute the covariance matrix $\boldsymbol{\Sigma}_k$ using Equation
  \eqref{eq:covc}.
\end{itemize}
\end{itemize}
Different adaptive algorithms have been proposed as improved versions
of the AM algorithm above. Good surveys of such
algorithms are found in \citet{Andrieu+Thoms:2008,Liang+Liu+Carroll:2010}, where the
authors present ways to implement the algorithms and then show why the
algorithms preserve the correct stationary distributions.
For instance, apart from updating the covariance alone, one can adapt $\lambda$ using the following
Robbins--Monro algorithm, which alleviates the problem of
$\boldsymbol{\Sigma}_k$ being systematically too large or too small
\citep{Andrieu+Thoms:2008,Atchade+Fort:2010,vihola:2011}:
\begin{equation}\label{eq:rma}
\log(\lambda_{k}) = \log(\lambda_{k-1}) + \gamma_{k} \, (\alpha_k-\overline{\alpha}).
\end{equation}
In Equation~\eqref{eq:rma}, $\overline{\alpha}$ is the target
acceptance rate which is commonly set to 0.234 and $\gamma_{k}$ is
a gain factor sequence satisfying the following conditions:
$$ \sum_{k=1}^\infty \gamma_k = \infty \quad \text{and} \quad
\sum_{k=1}^\infty \gamma_k^{1+\delta} < \infty \quad \text{ for some }
\delta\in(0,1].
$$
%
%\comment{In this article, we choice $\gamma_k=1/k$.}
%
In a recent paper, \citet{Vihola:2012} introduced the robust adaptive
Metropolis (RAM) algorithm with an online rule for adapting the covariance
matrix and a mechanism for maintaining the mean acceptance rate at a
pre-determined level.  \citet{Holden+Hauge+Holden:2009} describe an
adaptive independent Metropolis--Hastings algorithm, where the
proposal is adapted with past samples.  The limitation of this
adaptation is that the information gained from doing the local steps
cannot be used, so the algorithm iterations do not improve the
proposal \citep{Holden+Hauge+Holden:2009,Liang+Liu+Carroll:2010}.
\citet{Gilks+Roberts+Sahu:1998} proposed regeneration-based adaptive
algorithms, where after each regeneration point the proposal
distribution is modified based on all the past samples and the future
outputs become independent of the past. In the population-based
adaptive algorithms the proposal distributions are designed such that
computational techniques are incorporated into simulations and the
covariance matrix is adapted using a population of independent and
identically distributed samples from the adaptive direction sampler
\citep{Gilks+Roberts+George:1994} or the evolutionary Monte Carlo
\citep{Ren+Ding+Liang:2008}.

Another type of adaptive MCMC is proposed
by \citet{Vrugt+Braak+Diks+Robinson+Hyman+Higdon:2009} and
\citet{Vrugt+Braak:2011}, where they integrate the MCMC algorithm and
differential evolution. This type of algorithm generates multiple
different chains simultaneously for global exploration, and
automatically tunes the scale and orientation of the proposal
distribution in randomized subspaces during the search.
One way of adapting the MCMC proposal distribution is by using multiple
copies of the target density and Gibbs sampling
\cite{Cai+Meyer+Perron:2006,Griffin+Walker:2013}. The idea is
that a product of the proposal density and copies of the target density
is used to define a joint density which is sampled by Metropolis--Hastings-within-Gibbs algorithm.

When constructing an adaptive algorithm one should be sure that
the conditions for ergodicity are satisfied to ensure that the
algorithm converges to the target distribution
\citep{Andrieu+Moulines:2006,Roberts+Rosenthal:2007b,Saksman+Vihola:2010,Atchade+Fort:2010,Bai+Roberts+Rosenthal:2011}.
One way to ensure the ergodicity property of adaptive MCMC is in terms
of the two general conditions \citep{Bai:2009,Bai+Roberts+Rosenthal:2011}: diminishing adaptation
and containment condition. However,
these conditions are not necessary and there exist valid adaptation
mechanisms that do not have these properties.
%In our proposed
%algorithm, we use these conditions, but it might be possible to partially relax them.
In our proposed algorithm, to ensure ergodicity, we impose slightly stronger conditions that imply the diminishing 
and containment conditions. The imposed conditions also imply a strong law of large numbers rather than just a weak law of large numbers.

\section{Adaptive Metropolis Algorithm with Variational Bayesian Adaptive Kalman Filter Based Covariance Update}\label{sec:vbak}
In this section, we first briefly review the noise adaptive Kalman
filter \citep{Sarkka+Hartikainen:2013} which is used to adapt the
covariance matrix of the proposal distribution, and then present the
proposed variational Bayesian adaptive Metropolis (VBAM) algorithm.
\subsection{Kalman Filter}\label{sec:kf}
Kalman filter \citep{Kalman:1960} is the classical algorithm for
estimation of the dynamic state from noisy measurements in linear
Gaussian state space models. In probabilistic terms
the model can be expressed as \citep{Jazwinski:1970,Sarkka:2013}
\begin{align*}
\mathbf{x}_k& \sim
  \mathbb{N}(\mathbf{A}_{k-1} \, \mathbf{x}_{k-1}, \mathbf{Q}_{k-1}), \\
  \mathbf{y}_k&
  \sim\mathbb{N}(\mathbf{H}_k \, \mathbf{x}_k,\boldsymbol{\Sigma}_k),
\end{align*}
where $\mathbb{N}(\cdot)$ denotes the multivariate Gaussian
distribution, $\mathbf{x}_k \in \mathbb{R}^n$ is the dynamic state,
$\mathbf{A}_{k-1}$ is the dynamic model matrix, $\mathbf{Q}_{k-1}$ is
the process noise covariance, $\mathbf{y}_k \in \mathbb{R}^d$ is the
measurement, $\mathbf{H}_k$ is the measurement matrix, and
$\boldsymbol{\Sigma}_k$ is the measurement noise covariance matrix.
Here, $\mathbf{x}_k$ is an unknown variable and
$\mathbf{y}_k$ is an observed variable, whereas the matrices $\mathbf{A}_{k-1}$, $\mathbf{Q}_{k-1}$, $\mathbf{H}_{k}$, and $\boldsymbol{\Sigma}_k$ are assumed known.
We further assume that
$\mathbf{x}_0\sim\mathbb{N}(\mathbf{m}_0,\mathbf{P}_0)$, where
$\mathbf{m}_0$ and $\mathbf{P}_0$ are the known prior mean and
covariance. The estimation of states is recursively performed using
two Kalman filter steps:
\begin{enumerate}
\item Prediction step:
\begin{equation}
\begin{split}
\mathbf{m}_k^- &= \mathbf{A}_{k-1} \, \mathbf{m}_{k-1}, \\
\mathbf{P}_k^- &= \mathbf{A}_{k-1} \,  \mathbf{P}_{k-1} \, \mathbf{A}^T_{k-1}
  + \mathbf{Q}_{k-1}.
\end{split}
\end{equation}

\item Update step:
\begin{equation}
\begin{split}
  \mathbf{S}_k &=
  \mathbf{H}_k \, \mathbf{P}_k^- \, \mathbf{H}_k^T + \boldsymbol{\Sigma}_k, \\
\mathbf{K}_k &= \mathbf{P}_k^- \, \mathbf{H}_k^T \, \mathbf{S}_k^{-1}, \\
\mathbf{m}_k &= \mathbf{m}_k^- + \mathbf{K}_k \,
  \left(\mathbf{y}_k - \, \mathbf{H}_k \, \mathbf{m}_k^-\right), \\
\mathbf{P}_k &= \mathbf{P}_k^- - \mathbf{K}_k \, \mathbf{S}_k \, \mathbf{K}_k^T,
\end{split}
\end{equation}
\end{enumerate}
where $\mathbf{m}_k^-$ is the a priori state mean, $\mathbf{m}_k$ is
the a posteriori state mean, $\mathbf{P}_k^-$ is the a priori state
covariance, and $\mathbf{P}_k$ is the a posteriori state
covariance. In Bayesian sense the Kalman filter computes the
statistics for the following conditional distribution of the state
given the measurements:
\begin{equation}
  p(\mathbf{x}_k \mid \mathbf{y}_{1:k})
  = \mathbb{N}(\mathbf{x}_k \mid \mathbf{m}_k,\mathbf{P}_k).
\end{equation}
If $\mathbf{A}_{k-1} = \mathbf{I}$ and the process noise is zero
$\mathbf{Q}_{k-1} = \mathbf{0}$, the Kalman filter reduces to the so called
recursive least squares (RLS) algorithm which solves a general
multivariate linear estimation (regression) problem recursively.  The
matrices $\mathbf{A}_{k-1}$ and $\mathbf{Q}_{k-1}$ can be used for
modeling the dynamics of the state when it is not assumed to be static
(as in RLS).
\subsection{Variational Noise Adaptation}
In the Kalman filter discussed above, the model matrices in the state
space model are assumed to be known. The variational Bayesian adaptive
Kalman filter (VB-AKF, \citep{Sarkka+Nummenmaa:2009,Sarkka+Hartikainen:2013}) is considered with
the case where the noise covariance $\boldsymbol{\Sigma}_k$ is unknown.
The model is assumed to be of the form:
\begin{align}
  \mathbf{x}_k &\sim
    \mathbb{N}(\mathbf{A}_{k-1} \, \mathbf{x}_{k-1}, \mathbf{Q}_{k-1}),\label{eq:001}\\
  \mathbf{y}_k &\sim
    \mathbb{N}(\mathbf{H}_k\mathbf{x}_k,\boldsymbol{\Sigma_k}),\label{eq:002}\\
  \boldsymbol{\Sigma}_k &\sim
  p(\boldsymbol{\Sigma}_k  \mid  \boldsymbol{\Sigma}_{k-1})\label{eq:003},
\end{align}
where Equation \eqref{eq:003} defines the Markovian dynamic model
prior for the unknown measurement noise covariances. If we were
able to implement the optimal (non-Gaussian) Bayesian filter for this model,
it would compute the distribution
\begin{align}
  p(\mathbf{x}_k,\boldsymbol{\Sigma}_k \mid \mathbf{y}_{1:k}).
\end{align}
Recall that the Kalman filter can be considered as a generalization of
the RLS algorithm. In the same way we can
consider the Bayesian filter for the above model as a generalization
of the RLS algorithm where the noise variance is estimated together
with the linear regression solution. We can also interpret the AM
covariance adaption rule in Equation \eqref{eq:covc} as a simple
linear regression problem where we estimate the noise covariance
together with the linear regression solution. However, in AM we throw
out the linear regression solution and only retain the covariance.

Because the above model is a generalization of RLS with covariance
estimation, it can be seen to provide a recursive solution for estimation
of the covariance in Equation \eqref{eq:covc} as a special case. This
is the idea of our method. However, there is no reason to only
use the parameters which reduce the Bayesian filter to the RLS, but
one can use the full state space model and the corresponding filter
to construct an adaptive Metropolis algorithm.

Unfortunately, the exact Bayesian filter for the above model is
computationally intractable. However, the joint filtering distribution of
the state and covariance matrix can be approximated with the
free-form variational Bayesian approximation as follows
\citep{Sarkka+Nummenmaa:2009,Sarkka+Hartikainen:2013}:
\begin{equation}
p(\mathbf{x}_k,\boldsymbol{\Sigma}_k \mid \mathbf{y}_{1:k})
\approx Q_\mathbf{x}(\mathbf{x}_k) \, Q_{\boldsymbol{\Sigma}}(\boldsymbol{\Sigma}_k),
\end{equation}
where $Q_\mathbf{x}(\mathbf{x}_k)$ and
$Q_{\boldsymbol{\Sigma}}(\boldsymbol{\Sigma}_k)$ are unknown
approximating densities formed by minimizing the Kullback--Leibler
(KL) divergence between the true distribution and the approximation:
\begin{equation}
\begin{split}
&\textrm{KL}\left[Q_\mathbf{x}(\mathbf{x}_k)Q_{\boldsymbol{\Sigma}}(\boldsymbol{\Sigma}_k)
\,||\,p(\mathbf{x}_k,\boldsymbol{\Sigma}_k \mid \mathbf{y}_{1:k})\right] \\
&\qquad = \int Q_\mathbf{x}(\mathbf{x}_k) \, Q_{\boldsymbol{\Sigma}}(\boldsymbol{\Sigma}_k)
  \, \log \left(
  \frac{Q_\mathbf{x}(\mathbf{x}_k) \, Q_{\boldsymbol{\Sigma}}(\boldsymbol{\Sigma}_k)}
       {p(\mathbf{x}_k,\boldsymbol{\Sigma}_k \mid \mathbf{y}_{1:k})}
   \right) \, \mathrm{d}\mathbf{x} \, \mathrm{d}\boldsymbol{\Sigma}_k.
\end{split}
\label{eq:kldiv}
\end{equation}
The KL divergence above can be minimized with respect
to $Q_\mathbf{x}(\mathbf{x}_k)$ and
$Q_{\boldsymbol{\Sigma}}(\boldsymbol{\Sigma}_k)$ using the methods
from calculus of variations, which results in the following
approximations \citep{Sarkka+Nummenmaa:2009,Sarkka+Hartikainen:2013}:
\begin{align}
Q_\mathbf{x}(\mathbf{x}_k)
  &\propto
  \exp\left(\int\log p(\mathbf{x}_k,\boldsymbol{\Sigma}_k \mid \mathbf{y}_{1:k}) \,
     Q_{\boldsymbol{\Sigma}}(\boldsymbol{\Sigma}_k) \,
     \mathrm{d}\boldsymbol{\Sigma}_k\right), \label{eq:qx_eq} \\
  Q_{\boldsymbol{\Sigma}}(\boldsymbol{\Sigma}_k)
  &\propto
  \exp\left(\int\log p(\mathbf{x}_k,\boldsymbol{\Sigma}_k \mid \mathbf{y}_{1:k}) \,
  Q_{\mathbf{x}}(\mathbf{x}_k) \,
  \mathrm{d}\mathbf{x}_k\right). \label{eq:qs_eq}
\end{align}
Solving the above equations leads to the following
approximation \citep{Sarkka+Nummenmaa:2009,Sarkka+Hartikainen:2013}:
\begin{align}
  p(\mathbf{x}_k,\boldsymbol{\Sigma}_k \mid \mathbf{y}_{1:k-1})
  &\approx \mathbb{N}(\mathbf{x}_k \mid,\mathbf{m}_k,\mathbf{P}_k) \,
           \mathbb{IW}(\boldsymbol{\Sigma}_k \mid \nu_k,\mathbf{V}_k),
\end{align}
where $\mathbf{m}_k$ and $\mathbf{P}_k$ are given by the standard Kalman filter, and $\nu_k$ and
$\mathbf{V}_k$ are the parameters of the inverse
Wishart ($\mathbb{IW}$) distribution. The proposal covariance can, for
example, be computed as the mean of the inverse Wishart distribution:
  \begin{equation}\label{eq:sigma}
    \boldsymbol{\Sigma}_k=\frac{1}{\nu_k-d-1}\mathbf{V}_k.
  \end{equation}
The dynamic model $p(\boldsymbol{\Sigma}_k\mid\boldsymbol{\Sigma}_{k-1})$ needs to be chosen such that it produces inverse Wishart distribution on the Bayesian filter prediction step. As pointed out by
\citet{Sarkka+Nummenmaa:2009}, this kind of dynamical model is hard to construct explicitly, and hence they proposed heuristic dynamics for the covariances, which was then extended by \citet{Sarkka+Hartikainen:2013}. The following dynamic model is obtained:
\begin{equation}
\begin{split}
\nu_k^-& =\rho(\nu_{k-1}-d-1)+d+1, \\
%\mathbf{V}_k^-&=(\nu_{k-1}-d-1)\mathbf{B}\boldsymbol{\Sigma}_{k-1}\mathbf{B}^T,
\boldsymbol{\Sigma}_k^-&=
  \mathbf{B} \, \boldsymbol{\Sigma}_{k-1} \, \mathbf{B}^T,
\end{split}
\end{equation}
where $\nu_k^-$ and $\boldsymbol{\Sigma}_k^-$ are prior parameters, $\rho$ is a real number $0<\rho\leq 1$ and $\mathbf{B}$ is a matrix $0<|\mathbf{B}|\leq 1$.
Here the parameter $\rho$ controls the forgetting of the previous estimates of the covariance matrix by decreasing the degrees of freedom exponentially. The matrix $\mathbf{B}$ can be used to model the deterministic dynamics of the covariance matrix.

For our purposes, it is useful to write the VB-AKF algorithm in
a slightly modified form from \citet{Sarkka+Hartikainen:2013}, such that it involves the
covariance matrix in \eqref{eq:sigma} explicitly. The results is Algorithm~\ref{alg:vbakf}.
Note that in the algorithm, the matrices $\mathbf{A}_k$, $\mathbf{Q}_k$ and $\mathbf{H}_k$ are defined
by the selected state space model, and their use in VBAM will be shown in the numerical examples section.
%
%%%%
\begin{algorithm}[H]
\caption{VB-AKF algorithm}
\label{alg:vbakf}
\begin{itemize}
\item Initialize $\nu_0$, $\mathbf{m}_0$, $\mathbf{P}_0$ and $\boldsymbol{\Sigma}_0$.
\item For $k=1,2,\ldots$
\begin{itemize}
\item \textbf{Prediction}: compute the parameters of the predicted distribution:
    \begin{align*}
\mathbf{m}_k^- &= \mathbf{A}_{k-1} \, \mathbf{m}_{k-1}, \\
\mathbf{P}_k^- &= \mathbf{A}_{k-1} \,  \mathbf{P}_{k-1} \, \mathbf{A}^T_{k-1}
  + \mathbf{Q}_{k-1}, \\
\nu_k^-&=\rho(\nu_{k-1}-d-1)+d+1, \\
%\mathbf{V}_k^-&=(\nu_{k-1}-d-1)\mathbf{B}\boldsymbol{\Sigma}_{k-1}\mathbf{B}^T,
\boldsymbol{\Sigma}_k^-&=
  \mathbf{B} \, \boldsymbol{\Sigma}_{k-1} \, \mathbf{B}^T,
    \end{align*}

  \item \textbf{Update}: set $\nu_k=\nu_k^-+1$ and
    $\boldsymbol{\Sigma}_k^{(1)} = \boldsymbol{\Sigma}_k^-$. Iterate
    the following until the convergence (say, $N$ times for
    $j=1,\ldots,N$):
    \begin{align}
%\mathbf{R}_k^{(j+1)}&=\frac{1}{\nu_k-d-1}\mathbf{V}_k^{(j)}\\
\mathbf{S}_k^{(j+1)}&=\mathbf{H}_k \, \mathbf{P}_k^- \, \mathbf{H}_k^T
  + \boldsymbol{\Sigma}_k^{(j)}, \nonumber\\
\mathbf{K}_k^{(j+1)}&=\mathbf{P}_k^- \, \mathbf{H}_k^T \,
  \left(\mathbf{S}_k^{(j+1)}\right)^{-1}, \nonumber\\
\mathbf{m}_k^{(j+1)}&=\mathbf{m}_k^- + \mathbf{K}_k^{(j+1)} \,
  \left(\mathbf{y}_k - \mathbf{H}_k \, \mathbf{m}_k^-\right), \nonumber\\
\mathbf{P}_k^{(j+1)}&=\mathbf{P}_k^- - \mathbf{K}_k^{(j+1)} \,
  \mathbf{S}_k^{(j+1)} \, \left(\mathbf{K}_k^{(j+1)}\right)^T, \nonumber\\
%\boldsymbol{\Sigma}_k^{(j+1)} &=  \frac{1}{\nu_{k}-d-1}\mathbf{V}_k^-
%    + \frac{1}{\nu_{k}-d-1}\Xi_k^{(j+1)}\Xi_k^{(j+1)^T}\nonumber\\
%  &+ \frac{1}{\nu_{k}-d-1} \mathbf{H}_k\mathbf{P}_k\mathbf{H}^T_k\label{eq:vk}
\boldsymbol{\Sigma}_k^{(j+1)} &=
  \left( \frac{\nu_{k-1}-d-1}{\nu_{k}-d-1} \right) \boldsymbol{\Sigma}_k^-
    + \left( \frac{1}{\nu_{k}-d-1} \right)
  \mathbf{H}_k \, \mathbf{P}_k^{(j+1)} \, \mathbf{H}^T_k
       \nonumber\\
   &+ \left( \frac{1}{\nu_{k}-d-1} \right)
       \left( \mathbf{y}_k - \mathbf{H}_k \, \mathbf{m}_k^{(j+1)} \right)
      \left( \mathbf{y}_k - \mathbf{H}_k \, \mathbf{m}_k^{(j+1)} \right)^T.
       \label{eq:vk}
    \end{align}

\item Set $\boldsymbol{\Sigma}_k=\boldsymbol{\Sigma}_k^{(N)}$, $\mathbf{m}_k=\mathbf{m}_k^{(N)}$ and $\mathbf{P}_k=\mathbf{P}_k^{(N)}$.

\end{itemize}
\end{itemize}
\end{algorithm}
In Algorithm~\ref{alg:vbakf}, the choice of the number of iterations $N$ depends on the problem at hand. However, in the numerical examples we tested, we found out that the algorithm requires only a few iterations to converge (we used $N=5$). However, it would also be possible to use a stopping criterion which determines a suitable time to stop by monitoring the changes in the estimates at each iteration.
In the next section we will use the above algorithm to compute the covariance matrix
$\boldsymbol{\Sigma}_k$ for a proposal distribution.
\subsection{Variational Bayesian Adaptive Metropolis Algorithm}\label{sec:amcmcvb}
We are now ready to describe the variational Bayesian adaptive
Metropolis (VBAM) algorithm where the covariance matrix is updated with
VB-AKF. The idea is simply to replace the covariance computation in
Equation \eqref{eq:covc} with the estimate of covariance computed by
VB-AKF Algorithm \ref{alg:vbakf}. However, to ensure the convergence
of the method, we need the following restrictions (see
Section~\ref{sec:lln}):
\begin{enumerate}
\item The state space model in Equations~\eqref{eq:001} and
  \eqref{eq:002} needs to be uniformly completely
  controllable and observable with any
  bounded sequence of $\boldsymbol{\Sigma}_k$.
  This is quite natural, because otherwise the state estimation
  problem would not make sense as such. For the
  definitions of uniformly completely controllable and observable models, see Definitions
  \ref{def:uco} and \ref{def:ucc}.

\item We need to have $\rho = 1$ and $\mathbf{B} = \mathbf{I}$ in our
  dynamic model for the covariance matrix. This is needed to enforce
  diminishing adaptation of the covariance matrix.

\item The target distribution needs to be compactly supported. In
  principle, this is a restriction on the application domain of the
  method, but in practice, we can imagine to truncate the distribution
  at some value which exceeds to maximum floating point number that
  can be represented in the computer system.

\item We need to force uniform lower and upper bounds for the
  covariances $\boldsymbol{\Sigma}_k$ and therefore we need to include
  an additional boundary check to the method. For that purpose, we
  fix some constants $0<\mu_{1}\le \mu_{2}<\infty$ determining the
  feasible values for $\boldsymbol{\Sigma}_k$, and enforce by
  projection whenever necessary that $\mu_1 \, \mathbf{I} \le
  \boldsymbol{\Sigma}_k \le \mu_{2} \, \mathbf{I}$.  In practice, we
  can set $\mu_{1}$ to be very small and $\mu_{2}$ very high so that
  we practically never hit the boundaries.
  Note that for two matrices $A$ and $B$, the ordering $A \leq B$
  here means that $B - A \geq 0$ is positive semidefinite.

\item We also ensure that the values $\lambda_k$ stay within
  $[\delta_\lambda,\delta_\lambda^{-1}]$ for some constant
  $\delta_\lambda\in(0,1]$ by using a truncation procedure.
\end{enumerate}
The following Algorithm~\ref{alg:ama} is the variational Bayesian
adaptive Metropolis (VBAM) algorithm.
\begin{algorithm}[H]
\caption{Adaptive Metropolis algorithm with VB-AKF algorithm}
\label{alg:ama}
\begin{itemize}
\item Initialize $\boldsymbol{\theta}_0$,
  $\boldsymbol{\Sigma}_0$, $\mathbf{m}_0$, $\mathbf{P}_0$, $\nu_0$,
  and $\lambda_0$.
  The choice of initial parameters depends on the problem at hand. However, note that
  $\mathbf{m}_0 = \mathbf{0}$, $\mathbf{P}_0 = \mathbf{I}$, $\lambda_0=2.38^2/d$ and $\nu_0 = d+2$ are often used.
\item For $k=1,2,\ldots$
\begin{itemize}
\item Sample the candidate from the Gaussian proposal distribution
  $\boldsymbol{\theta}_*\sim\mathbb{N}(\boldsymbol{\theta}_{k-1},\lambda_{k-1}
  \,\boldsymbol{\Sigma}_{k-1})$ or from a Student's
  $t$-distribution with location $\theta_{k-1}$ and scale $\lambda_{k-1}
  \,\boldsymbol{\Sigma}_{k-1}$.
\item Calculate the acceptance probability:
$$\alpha_k=\text{min}\left\{1,\, \frac{p(\boldsymbol{\theta}_* \mid \mathbf{z}_1,\ldots,
\mathbf{z}_M)}{p(\boldsymbol{\theta}_{k-1} \mid \mathbf{z}_1,\ldots,\mathbf{z}_M)}\right\}.
$$
\item Generate $u\sim\mathbb{U}(0,1)$
\item Set
\[ \boldsymbol{\theta}_k = \left\{ \begin{array}{ll}
\boldsymbol{\theta}_* & \text{if $u<\alpha_k$},\\
\boldsymbol{\theta}_{k-1}& \text{Otherwise}.\end{array} \right. \]

\item Update the proposal covariance matrix by computing it with the
  VB-AKF Algorithm \ref{alg:vbakf} where
  $\mathbf{y}_k=\boldsymbol{\theta}_k$. Check that $\boldsymbol{\Sigma}_k >
  \mu_{1} \mathbf{I}$ and $\boldsymbol{\Sigma}_k < \mu_{2}
  \mathbf{I}$. If this is not true, set $\boldsymbol{\Sigma}_k =
  \boldsymbol{\Sigma}_{k-1}$ and do a single iteration of the VB-AKF
  update step (ignoring the last equation) to compute the updated mean
  and covariance corresponding to this noise covariance.

\item Optionally, update $\lambda_k$ using Equation~\eqref{eq:rma}.
  If $\lambda_k \notin [\delta_\lambda,\delta_\lambda^{-1}]$, then
  force it to the interval $[\delta_\lambda,\delta_\lambda^{-1}]$.
\end{itemize}
\end{itemize}
\end{algorithm}
%%%%%%%%%%%%%%%%%%%%%%%%%%
%
%
\section{Proof of Convergence of VBAM}\label{sec:lln}
In this section we show the convergence of our variational Bayesian
adaptive Metropolis algorithm (VBAM) using the law of large numbers.
We consider two variants of the algorithm, with and without
$\lambda_k$, simultaneously. To prove the convergence of our proposed
VBAM algorithm, it is enough to prove the following theorem.

\begin{theorem}
    \label{th:slln}
    Suppose that the target density
    $\pi(\boldsymbol{\theta})=p(\boldsymbol{\theta} \mid
    \mathbf{z}_1,\ldots,\mathbf{z}_M)$ is bounded and has a bounded
    support. Furthermore, suppose that $\sum_{k\ge 1} k^{-1} \gamma_k
    <\infty$.  Then, for bounded functions $f$, the strong law of
    large numbers holds,
\[
    \frac{1}{n} \sum_{k=1}^n f(\mathbf{\boldsymbol{\theta}}_k) \xrightarrow{\text{a.s.}}
    \int_{\R^d}
    f(\boldsymbol{\theta}) \pi(\boldsymbol{\theta}) \ud \boldsymbol{\theta},
\]
where $\boldsymbol{\theta}_1,\ldots,\boldsymbol{\theta}_n$ are VBAM
generated samples assuming the uniform complete controllability and
observability properties (see Definitions \ref{def:uco} and
\ref{def:ucc}) of the state space model are satisfied.
\end{theorem}
Before proving Theorem~\ref{th:slln}, we introduce some concepts in
the filtering theory and some supporting lemmas that are used in the
proof. We briefly introduce the concepts of information matrix and
controllability matrix, as well as the concepts of uniform complete
observability and uniform complete controllability in terms of them.
These concepts are very well known in the field of statistical estimation
and filtering of stochastic processes \citep{Jazwinski:1970}, and they
ensure the conditions that the prior distribution is non-degenerate and
that the posterior covariance of the state is bounded and hence computing
an estimate of the state is possible in statistical sense.
\begin{definition}[Information matrix]
The information matrix $\digamma$ is defined as \citep{Jazwinski:1970}
\begin{eqnarray}
\digamma(M,M_0)=\sum_{k=M_0}^M
  \Phi^{T}(k,M) \, \mathbf{H}_k^T \,
\boldsymbol{\Sigma}_k^{-1} \, \mathbf{H}_k \,
  \Phi(k,M),
\end{eqnarray}
where $\Phi(k,M) = \mathbf{A}_{k-1}^{-1} \,\mathbf{A}_{k}^{-1} \, \mathbf{A}_{k+1}^{-1} \, \cdots \, \mathbf{A}_{M-2}^{-1} \, \mathbf{A}_{M-1}^{-1}$, for $k\leq M$.
\end{definition}
\begin{definition}[Controllabilty matrix]
The controllability matrix $\beth$  is defined as \citep{Jazwinski:1970}
\begin{eqnarray}
\beth(M,M_0)=\sum_{k=M_0}^{M}\Phi(M,k) \,
\mathbf{Q}_k \,\Phi^T(M,k),
\end{eqnarray}
where $\Phi(M,k) = \mathbf{A}_{M-1} \,\mathbf{A}_{M-2} \, \mathbf{A}_{M-3} \, \cdots \, \mathbf{A}_{k+1} \, \mathbf{A}_{k} \, \mathbf{A}_{k-1}$, for $k\leq M$.
\end{definition}
The state space model~\eqref{eq:001} and~\eqref{eq:002} is said to be
completely observable if $\digamma(M,1)>0$.  Similarly, it is
completely controllable if $\beth(M,1)>0$.
We now introduce the concepts of uniform complete observability and
uniform complete controllability as defined by \citep{Jazwinski:1970}.
These concepts are required in Lemma~\ref{lem:boundedness}.
\begin{definition}[Uniform complete observability] \label{def:uco}
  A system is uniformly completely observable if there exist a positive
  integer $L$, and constants $\beta_1, \beta_2 > 0$ such
  that for all $M\geq L$ we have
\begin{eqnarray}
  \beta_1 \, \mathbf{I}\leq \digamma(M,M-L) \leq \beta_2 \, \mathbf{I}.
\end{eqnarray}
\end{definition}
\begin{definition}[Uniform complete controllability] \label{def:ucc}
  A system is uniformly completely controllable if there exist a positive
  integer $L$, and constants $\beta_1, \beta_2 > 0$ such
  that for all $M\geq L$ we have
\begin{eqnarray}
\beta_1 \, \mathbf{I} \leq \beth(M,M-L) \leq \beta_2 \, \mathbf{I}.
\end{eqnarray}
\end{definition}
Having introduced the two important concepts for the proof, we next
state and prove the lemma which shows the boundedness of
$\mathbf{m}_k$ and $\mathbf{P}_k$. After that we state and prove the
lemma for diminishing of the difference of VBAM covariance matrices.
The two Lemmas \ref{lem:boundedness} and \ref{lem:sigmadifference} are
required in the proof of Theorem \ref{th:slln}.
\begin{lemma}
    \label{lem:boundedness}
    Assume that the state space system defined by
    Equations~\eqref{eq:001} and~\eqref{eq:002} is uniformly
    completely observable and uniformly completely controllable for
    any sequence $\boldsymbol{\Sigma}_k$ which is bounded below by
    $\mu_1\mathbf{I}$ and above by $\mu_2\mathbf{I}$. If
    $\mathbf{P}_0>0$, then $\mathbf{m}_k$ and $\mathbf{P}_k$ produced
    by the VBAM are uniformly bounded.
\end{lemma}
\begin{proof}
  The sequence $\boldsymbol{\Sigma}_k$ produced by VBAM algorithm is always
  bounded below by $\mu_1\mathbf{I}$ and above by $\mu_2\mathbf{I}$
  by construction.
  Given the bounded sequence $\boldsymbol{\Sigma}_k$, the computation
  of the means $\mathbf{m}_k$ and covariances $\mathbf{P}_k$ reduces
  to conventional Kalman filtering. Because the state space model is
  uniformly completely observable and controllable, by Lemmas 7.1 and
  7.2, and Theorem 7.4 in \citet{Jazwinski:1970}, the mean and
  covariance sequences are uniformly bounded provided that the measurements
  are bounded. The measurements are the samples from the target distribution
  which is assumed to be compactly supported and thus are bounded.
\end{proof}
The VB-AKF algorithm itself contains a fixed point iteration and for the
VBAM algorithm to converge to something sensible, we need this
iteration to converge. This is ensured by Theorem~\ref{thm:vbak} which
is proved below.
\begin{theorem}
\label{thm:vbak}
The sequence
$\{\boldsymbol{\Sigma}_k^{(1)},\boldsymbol{\Sigma}_k^{(2)},\ldots,\boldsymbol{\Sigma}_k^{(j-1)},\boldsymbol{\Sigma}_k^{(j)},\boldsymbol{\Sigma}_k^{(j+1)},\ldots\}$ produced by the VB-AKF algorithm iteration converges to a below and above bounded matrix $\Sigma_k^{(j)} \to \Sigma_k$ with $j \to \infty$ provided that the system is uniformly completely observable and controllable, $\nu_{k}-d-1 > 0$, $\mathbf{m}_{k-1}$, $\mathbf{P}_{k-1}$, and $\Sigma_{k-1}$ are bounded from above and below.
\end{theorem}
\begin{proof}
Recall \citep{Sarkka+Nummenmaa:2009,Sarkka+Hartikainen:2013} that Equations \eqref{eq:qx_eq} and \eqref{eq:qs_eq} actually arise in the solution to the minimization of the KL divergence functional
\begin{align*}
  \mathcal{J}[Q_{\mathbf{x}},Q_{\boldsymbol{\Sigma}}]
&= \textrm{KL}\left[Q_\mathbf{x}(\mathbf{x}_k) \, Q_{\boldsymbol{\Sigma}}(\boldsymbol{\Sigma}_k)
\,||\,p(\mathbf{x}_k,\boldsymbol{\Sigma}_k \mid \mathbf{y}_{1:k})\right], \\
&= \int Q_\mathbf{x}(\mathbf{x}_k) \, Q_{\boldsymbol{\Sigma}}(\boldsymbol{\Sigma}_k)
  \, \log \left(
  \frac{Q_\mathbf{x}(\mathbf{x}_k) \, Q_{\boldsymbol{\Sigma}}(\boldsymbol{\Sigma}_k)}
       {p(\mathbf{x}_k,\boldsymbol{\Sigma}_k \mid \mathbf{y}_{1:k})}
   \right) \, \mathrm{d}\mathbf{x} \, \mathrm{d}\boldsymbol{\Sigma}_k,
\end{align*}
under the constraints that $Q_\mathbf{x}$ and $Q_{\boldsymbol{\Sigma}}$ integrate to unity. The minimization with respect to $Q_\mathbf{x}$ then gives \eqref{eq:qx_eq} and the minimization with respect to $Q_{\boldsymbol{\Sigma}}$ gives \eqref{eq:qs_eq}.

The fixed point iteration in VB-AKF is just an implementation of the following coordinate descend iteration with $j=1,2,3,\ldots$:
\begin{align*}
	Q^{(j+1)}_{\mathbf{x}}
	&= \arg \min_{Q_{\mathbf{x}}}
	\mathcal{J}[Q_{\mathbf{x}},Q_{\boldsymbol{\Sigma}}^{(j)}], \\
	Q^{(j+1)}_{\boldsymbol{\Sigma}}
	&= \arg \min_{Q_{\boldsymbol{\Sigma}}}
	\mathcal{J}[Q_{\mathbf{x}}^{(j+1)},Q_{\boldsymbol{\Sigma}}].
\end{align*}
It is now easy to show that $\mathcal{J}$ is a convex functional in both the arguments and hence the coordinate descend is guaranteed to converge \citep{Luenberger:2008}. When the probability density of a Gaussian distribution converges, its mean and covariance converge as well. Similarly, the convergence of the probability density of an inverse-Wishart distribution implies that its parameters convergence as well. The given conditions on the means and covariances ensure that the posterior distribution is well-behaved, which further ensures that the limiting approximating distributions and their statistics are well-behaved (finite and non-zero) as well. Hence the result follows.
\end{proof}
Next we will show that $\boldsymbol{\Sigma}_k -
\boldsymbol{\Sigma}_{k-1}$ converges to zero provided that we have
$\rho=1$ and $\mathbf{B} =\mathbf{I}$. Intuitively, the selection $\rho=1$ ensures that $a_k = 1 / (\nu_{k}-d-1)$ converges to zero and $\mathbf{B} =\mathbf{I}$ ensures
that the prediction step does not alter the converged covariance
matrix.
\begin{lemma}
    \label{lem:sigmadifference}
The sequences $\mathbf{m}_k$ and $\mathbf{P}_k$ produced by the
VBAM satisfy
\[
    \| \boldsymbol{\Sigma}_k - \boldsymbol{\Sigma}_{k-1} \|_F
    \le c \, a_k,
    %\frac{c_{\boldsymbol{\Sigma}}}{\nu_k - d - 1},
\]
where $c$ is a constant and $\|\cdot\|_F$ is the Frobenius norm.
\end{lemma}
\begin{proof}
  We will start by considering the case where in Algorithm
  \ref{alg:ama}, we have resorted to truncation and have set
  $\boldsymbol{\Sigma}_k=\boldsymbol{\Sigma}_{k-1}$. Because this
  implies that $\|\boldsymbol{\Sigma}_k-\boldsymbol{\Sigma}_{k-1}\|_F=0$,
  the condition is trivially satisfied. In the case that
  the truncation is not done we proceed as follows.

  Let's denote $\Xi_k = \mathbf{y}_k-\mathbf{H} \, \mathbf{m}_k$.
  Since by Lemma~\ref{lem:boundedness}, $\boldsymbol{\Sigma}_k$,
  $\mathbf{m}_k$ and $\mathbf{P}_k$ are bounded, we have $\|
  \boldsymbol{\Sigma}_k \|_F \leq c_{\boldsymbol{\Sigma}}$, $\|\Xi_k
  \, \Xi_k^T\|_F\leq c_{\Xi}$ and $\|\mathbf{H}_k \, \mathbf{P}_k \,
  \mathbf{H}^T_k\|_F \leq c_{\mathbf{P}}$ for some constants
  $c_{\boldsymbol{\Sigma}}$, $c_{\Xi}$ and $c_{\mathbf{P}}$.
  According to Theorem~\ref{thm:vbak} the iteration in VB-AKF
  converges and thus we can assume $\boldsymbol{\Sigma}_k$ is the
  solution to the fixed point equation.  If we set $\mathbf{B} =
  \mathbf{I}$, Equation~\eqref{eq:vk} then gives
\begin{align}
	 \|\boldsymbol{\Sigma}_k - \boldsymbol{\Sigma}_{k-1}\|_F
    &=\|-a_k \, \boldsymbol{\Sigma}_{k-1} + a_k \, \Xi_k \, \Xi_k^T
  + a_k \, \mathbf{H}_k \, \mathbf{P}_k \, \mathbf{H}_k^T\|_F,\nonumber\\
%    &=a_k \, \|-\boldsymbol{\Sigma}_{k-1} + \Xi_k \, \Xi_k^T
%    + \mathbf{H}_k \, \mathbf{P}_k \, \mathbf{H}_k^T\|_F\nonumber\\
	&\leq a_k \, \left (\|
  \boldsymbol{\Sigma}_k \|_F + \|\Xi_k\Xi_k^T\|_F
  + \|\mathbf{H}_k\mathbf{P}_k\mathbf{H}^T_k\|_F \right),\nonumber\\
	&\leq a_k(c_{\boldsymbol{\Sigma}} + c_{\Xi}+c_{\mathbf{P}}). \nonumber \qedhere
\end{align}
\end{proof}
Next we state and prove the lemma which shows the boundedness of the difference
of transition probabilities. Let us denote by
$\Pi_{\lambda,\boldsymbol{\Sigma}}$
the Markov transition probability
of a random-walk Metropolis algorithm with the increment
proposal distribution $\mathbb{N}(0, \lambda \boldsymbol{\Sigma})$ .
\begin{lemma}
    \label{lem:kernel-diff}
    For the VBAM algorithm, there exists a constant $c_D<\infty$ such
    that
\[
    \| \Pi_{\lambda_k \boldsymbol{\Sigma}_k}
   - \Pi_{\lambda_{k-1} \boldsymbol{\Sigma}_{k-1}} \|
\defeq
\sup_{\boldsymbol{\theta}} \sup_{|f|\le 1} | \Pi_{\lambda_k \boldsymbol{\Sigma}_k} f(\boldsymbol{\theta}) -
    \Pi_{\lambda_{k-1} \boldsymbol{\Sigma}_{k-1}}
      f(\boldsymbol{\theta}) |
    \le c_D \, (a_k + \gamma_k),
    %\Big(\gamma_k + \frac{1}{\nu_k - d - 1}\Big),
\]
where $\Pi_{\lambda\boldsymbol{\Sigma}} f(\boldsymbol{\theta})
\defeq \int f(\theta') \,
\Pi_{\lambda\boldsymbol{\Sigma}}(\boldsymbol{\theta},\ud \theta')$,
and the latter supremum is taken with respect to measurable functions.
\end{lemma}
\begin{proof}
    By construction, the eigenvalues of $\lambda_k \, \boldsymbol{\Sigma}_k$
    are bounded and
    bounded away from zero, uniform in $k$.
    Therefore, Proposition 26(i) of \cite{vihola:2011} implies the
    existence of a constant $c_\Delta<\infty$ such that
\begin{align*}
| \Pi_{\lambda_k \, \boldsymbol{\Sigma}_k} f(\boldsymbol{\theta}) -
    \Pi_{\lambda_{k-1} \, \boldsymbol{\Sigma}_{k-1}}
      f(\boldsymbol{\theta}) |
%    |P_{\Theta_k}f(x) - P_{\Theta_{k-1}}f(x)|
    &\le c_\Delta \| \lambda_k \, \boldsymbol{\Sigma}_k
    - \lambda_{k-1} \, \boldsymbol{\Sigma}_{k-1} \|_F, \\
    &= c_\Delta
    \|(\lambda_k-\lambda_{k-1}) \, \boldsymbol{\Sigma}_{k}
  + \lambda_{k-1} \, (\boldsymbol{\Sigma}_k - \boldsymbol{\Sigma}_{k-1})\|_F,\\
    &\leq c_\Delta | \lambda_k - \lambda_{k-1}| \,
    \|\boldsymbol{\Sigma}_{k}\|_F
    + \lambda_{k-1} \, \|\boldsymbol{\Sigma}_k-\boldsymbol{\Sigma}_{k-1}\|_F.
\end{align*}
Both $\|\boldsymbol{\Sigma}_{k}\|_F$ and $\lambda_{k-1}$ are bounded,
and an easy computation shows that $|\lambda_k - \lambda_{k-1}| \le
c_\lambda \gamma_k$ for some constant $c_\lambda=c_\lambda(\beta_0,\beta_1)<\infty$,
because $(\lambda_k)_{k\ge 1}$ is
bounded. The proof is concluded by applying Lemma~\ref{lem:sigmadifference}
to bound $\|\boldsymbol{\Sigma}_k-\boldsymbol{\Sigma}_{k-1}\|_F$.
\end{proof}
Finally, we give the proof of Theorem~\ref{th:slln}.
\begin{proof}[Proof of Theorem~\ref{th:slln}]
We use Corollary 2.8 in \citet{Fort+Moulines+Priouret:2011},
with $\Theta_n = \lambda_n \, \boldsymbol{\Sigma}_n$,
$V\equiv 1$ and $\pi_{\Theta_n}=\pi$.
We need to check that the conditions (A3)--(A5) are
satisfied. It is easy to show that (A3) holds for example with
$\lambda_\Theta=b_\Theta = 1/2$, because the eigenvalues of feasible
covariance matrices $\Theta$
are bounded and bounded away from zero by construction,
and one can find $\delta_\Theta\ge \delta>0$ (see Theorem 7 in \citep{vihola:2011}).
Therefore, $\sup_\Theta L_\Theta<\infty$
and (A5) holds trivially. For (A4), it is enough to observe that
\[
    \sum_{k=1}^\infty k^{-1} \| \Pi_{\Theta_k} - \Pi_{\Theta_{k-1}} \|
    \le
    \sum_{k=1}^\infty k^{-1} c_P
    \Big(\gamma_k + a_k \Big)
    < \infty,
\]
where we have used Lemma~\ref{lem:kernel-diff}, our assumption on
$\gamma_k$ and the
fact that $\sup_{k\ge 1} k \, a_k<\infty$, because
$\nu_k$ increases linearly if we select
$\rho = 1$ which gives $\nu_k = \nu_0 + k$.
\end{proof}
\begin{remark}
  We note that our result generalises immediately if we replace the
  Gaussian proposal distribution with a multivariate Student's
  $t$-distribution.  It is possible to elaborate our result by
  omitting the bounds for the scaling adaptation by additional
  regularity conditions on $\pi$. This can be achieved by showing
  first the stability of the scaling adaptation 
  \cite{vihola:2011}.
\end{remark}

%
%
%%%%%
\section{Numerical Results} \label{sec:nume}
In this section, the convergence of the algorithm is assessed
empirically\footnote{The Matlab codes can be obtained from the
  corresponding author on request.}. We first present three simulated examples which are
often used in literature to study the performance and convergence of
adaptive MCMC algorithms. In the first example, we compare our proposed
VBAM algorithm with the AM algorithm proposed by
\citet{Haario+Saksman+Tamminen:1999,Haario+Saksman+Tamminen:2001}
using an example from the articles. We use a Gaussian random walk model as the state space model in the VBAM.
In the second example, we apply AM and VBAM to the
100-dimensional example from \citet{Roberts+Rosenthal:2009}.
In the third example, we apply VBAM to a well-known
benchmark of sampling from a 20-dimensional banana shaped distribution and compare
the empirical results with the AM results of \citet{Roberts+Rosenthal:2009}.
Finally, we apply the VBAM algorithm to two real data examples. In those examples
we analyze the chemical reaction model found in \citet{Himmelblau:1970}
and the bacteria growth model found in \citet{Brown+Berthouex:2002}.
\subsection{One-dimensional projection of the density function}
In this example, we consider a one-dimensional projection of 
the density function similar to that of \citet{Haario+Saksman+Tamminen:2001}. 
%The following example was also used in
%\citet{Haario+Saksman+Tamminen:2001} for
%demonstrating the performance of the AM algorithm. 
We aim to sample
from a density $\pi$ on the rectangle $R=[-18,18]\times[-3,3]\subset
\mathbb{R}^2$ as follows. Let $S=[-0.5,0.5]\times[-3,3]$ and set
\[ \pi(x) = \left\{ \begin{array}{ll}
1, & \text{if $x\in S$},\\
36,& \text{if $x\in R\setminus S$}.\end{array} \right. \]

For the AM algorithm, we initialized the proposal covariance as
$\boldsymbol{\Sigma}_0 = \mathbf{I}$ and updated it using
Equation~\eqref{eq:covc} with the values of $\epsilon$ and
$\lambda_k$ set to $0.0001$ and $2.8322$, respectively.
For the VBAM algorithm we used a random walk model defined as
\begin{equation}\label{eq:RWM}
\begin{split}
\mathbf{x}_k &= \mathbf{x}_{k-1} + \mathbf{q}_{k-1},\\
\mathbf{y}_k &= \mathbf{x}_{k} + \mathbf{r}_{k},
\end{split}
\end{equation}
where $\mathbf{A} = \mathbf{I}$ and $\mathbf{Q} = 0.001^2 \, \mathbf{I}$,
$\mathbf{H} = \mathbf{I}$. It is easy to show that this
model is uniformly completely observable and controllable for any
bounded sequence of $\boldsymbol{\Sigma}_k$. We initialized the prior
mean and covariance as $\mathbf{m}_0 = \begin{bmatrix} 0 & 0
\end{bmatrix}^T$ and $\mathbf{P}_0 = \mathbf{I}$ while we set
$\nu_0=4$.

\begin{figure}[htb]
\centering
\includegraphics{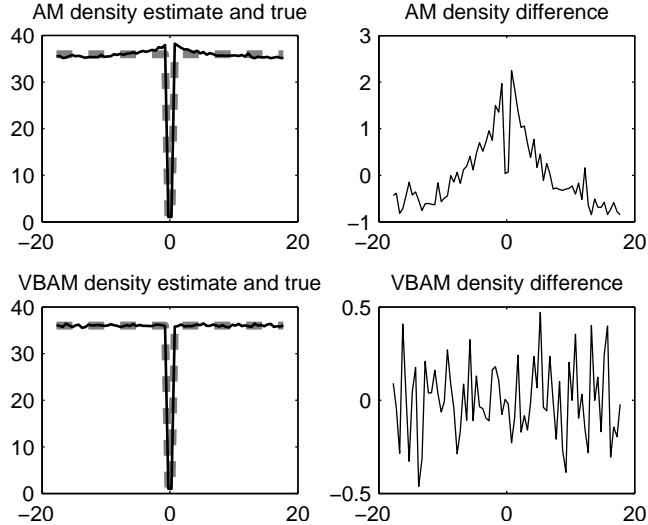}
\caption{Left column: the upper plot is the adaptive Metropolis (AM)
  density and the lower plot is the variational Bayesian adaptive
  Metropolis (VBAM) density.  The dashed line shows the true density
  while the black line shows the approximated density from samples.
  Right column: the upper plot shows the AM density difference between
  the true and approximated densities while the lower plot shows the
  density difference for the VBAM method.  The density difference for the
  VBAM algorithm fluctuates around zero while the AM density
  difference does not. Therefore VBAM algorithm performs better than
  AM.}
\label{fig:2}
\end{figure}

To compare the performance of AM and VBAM algorithms, we generated
$10^6$ samples from each algorithm. As in \citet{Haario+Saksman+Tamminen:2001}, the comparison is done through
computing the density differences of each method and the results are
shown in Figure~\ref{fig:2}. The left column plots show the empirical
densities produced by the adaptive Metropolis (AM) algorithm
and the proposed variational
Bayesian adaptive Metropolis (VBAM) algorithm, respectively.  The
right column shows, for each method, the difference between the real
target density and the empirical densities. The VBAM algorithm indeed
seems to give samples that follow well the true distribution, that is,
the empirical density approximates well the true density. Because the
density differences are around zero the results show the VBAM
algorithm performs better than the AM algorithm.

\subsection{100-dimensional Gaussian target distribution}
In this example, we consider a high-dimensional Gaussian target
distribution as discussed by \citet{Roberts+Rosenthal:2009}. Here, the
target distribution is $\mathbb{N}(0,MM^T)$, where the entries of the
$d\times d$ matrix $M$ are generated from unit Gaussian distributions
\citep{Roberts+Rosenthal:2009}. We compare our results with a version of
the AM algorithm considered in the article using the
C-language implementation (adaptchol.c) provided by the authors.
For the experiment we also implemented a C-language version of our VBAM
adaptation so that the only difference in the runs is the covariance adaptation.

In the AM algorithm of \citet{Roberts+Rosenthal:2009}, at iteration
$k$, the proposal distribution is defined as
\begin{equation}\label{eq:RR_proposal}
q_k(\theta_* \mid \theta) = \begin{cases}
\mathbb{N}\left(\theta,\frac{0.1^2}{d}\mathbf{I}\right), & \text{if $k\leq 2d$},\\
(1-\beta) \, \mathbb{N}\left(\theta,\frac{2.38^2}{d}\boldsymbol{\Sigma}_k\right)
+\beta \, \mathbb{N}\left(\theta,\frac{0.1^2}{d}\mathbf{I}\right),
&
\text{if $k > 2d$},\end{cases}
\end{equation}
where $\boldsymbol{\Sigma}_k$ is the current empirical covariance
and $\beta$ is a small positive constant. In the numerical experiment,
we chose the dimensionality to be $d=100$ and $\beta=0.05$ following
\citet{Roberts+Rosenthal:2009}. For the VBAM algorithm, the state space model used is the
Gaussian random walk model \eqref{eq:RWM} with $\mathbf{A} = \mathbf{I}$, $\mathbf{H} = \mathbf{I}$,
and $\mathbf{Q} = 10^{-9} \, \mathbf{I}$, and $\lambda_k$
is updated using
Equation~\eqref{eq:rma}. We initialized $\lambda_0=2.38^2/2$, the
value of $\overline{\alpha}$ was 0.234 and we used the following
$\gamma_k$ \citep{Liang+Liu+Carroll:2010}:
$$\gamma_k \, = \, \frac{k_0}{\max\{k_0,k^\tau\}},$$
where $k_0>1$ is the pre-specified value and $\tau\in(1/2 \,,\, 1]$. The
used values for $k_0$ and $\tau$ were 1000 and 0.99, respectively.

\begin{figure}[htb]
\centering
\includegraphics{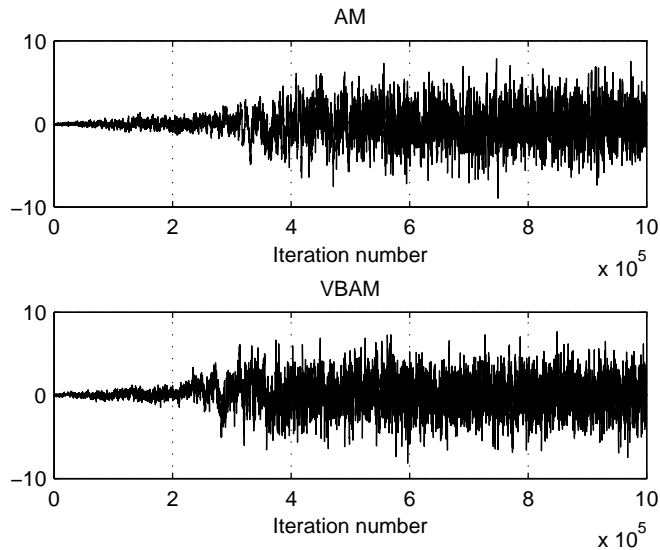}
\caption{First coordinate of the MCMC chains in the 100-dimensional Gaussian
target distribution example.}
\label{fig:x1_chains}
\end{figure}

\begin{figure}[htb]
\centering
\includegraphics{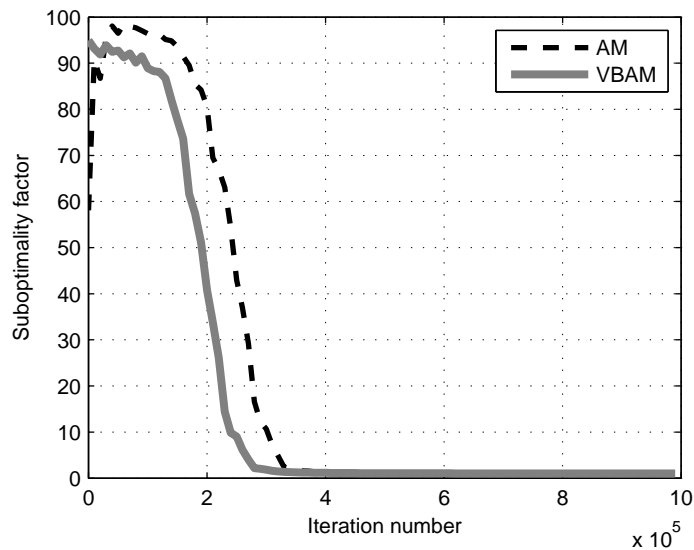}
\caption{Suboptimality factors in the Gaussian target distribution example.}
\label{fig:ratio}
\end{figure}

Trace plots of the first coordinate of the AM and VBAM chains are
shown in Figure~\ref{fig:x1_chains}. It can be seen that the VBAM
seems to stabilize to the stationary distribution a bit faster than
the AM algorithm. This is confirmed by the suboptimality factors
\citep{Roberts+Rosenthal:2009} shown in Figure~\ref{fig:ratio}.
In this examples, the suboptimality factor of VBAM reaches the value of around one
faster than AM which implies that its adaptation works faster.

\subsection{20-dimensional banana-shaped distribution}
In this example, we consider a banana-shaped distribution which is
also often used as an example to study the performance of adaptive MCMC
algorithms \citep{Haario+Saksman+Tamminen:1999,Roberts+Rosenthal:2009,Haario+Saksman+Tamminen:2001,Bornkamp:2011}.
The banana-shaped distribution density function
is given as
\begin{align*}
 \mathbf{f}(x_1,x_2,\ldots,x_d) &\propto \exp\left(-\frac{x_1^2}{200}-\frac{1}{2}(x_2+B\, x_1^2-100 \,B)^2-\frac{1}{2}(x_3^2+x_4^2+\cdots+x_d^2)\right),
\end{align*}
where $B>0$ is the `bananicity' constant.

\begin{figure}[htb]
\centering
\includegraphics{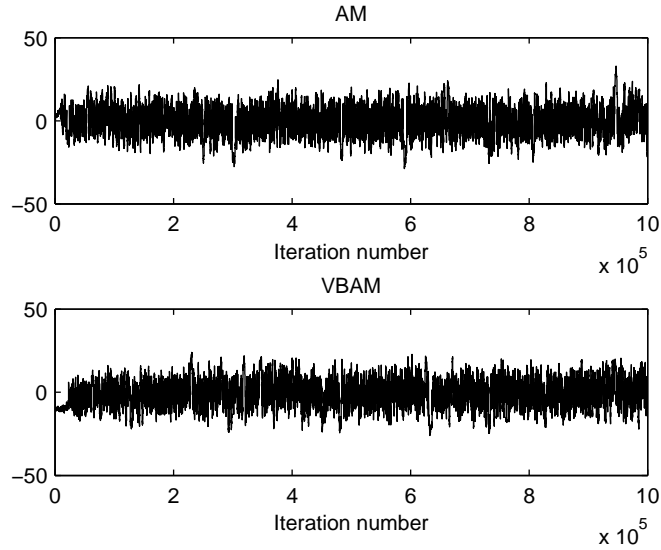}
\caption{Trace plots of the first component of the banana distribution. The first
  row is the AM plot and the second is for VBAM. The mixing of the samples
  in both AM and VBAM is good.}
\label{fig:banana_trace}
\end{figure}

We used the VBAM algorithm to sample from this distribution and compared the results
with the AM algorithm of \citet{Roberts+Rosenthal:2001} whose proposal distribution was \eqref{eq:RR_proposal}.
The VBAM state space model was the random walk model \eqref{eq:RWM} with
$\mathbf{A} = \mathbf{I}$, $\mathbf{H} = \mathbf{I}$ and $\mathbf{Q} = 10^{-9} \, \mathbf{I}$.
As in \citep{Roberts+Rosenthal:2001}, we set $d=20$ and $B=0.1$ and ran the
AM and VBAM algorithms for $10^6$ iterations. Figure~\ref{fig:banana_trace}
shows the trace plots for the first components $x_1$ and $x_2$. It can be observed that the AM and VBAM algorithms both
mix with an approximately same speed.

\begin{figure}[htb]
\centering
\includegraphics{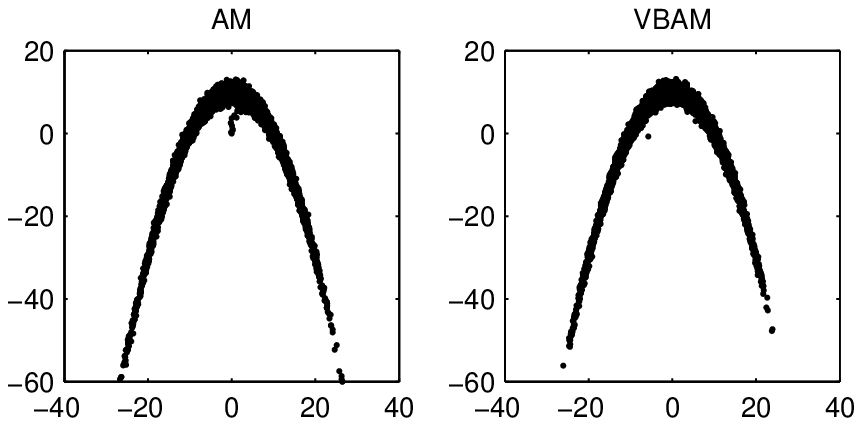}
\caption{Scatter plots for the first two components in the banana example. The first plot is for AM, the second plot is
for VBAM.}
\label{fig:banana}
\end{figure}

Figure~\ref{fig:banana} shows the scatter plots for AM and VBAM algorithms. As can be seen, the shapes of the plots
indeed have a banana-like shape and they cover the support of the
distribution well. In this example we did not find any significant difference between AM and VBAM, but
still it shows that also the VBAM algorithm works well in this challenging sampling problem.

\subsection{Chemical Reaction Model}
As the first real data example\footnote{Both the real data
models can also be found in \url{http://helios.fmi.fi/~lainema/mcmc/examples.html},
where the author analyzes them using MCMC.} we analyze a chemical reaction model
studied in the book by \citet{Himmelblau:1970}.
The model found on page 326--327 \citep{Himmelblau:1970} is a result of deterministic modeling of chemical reactions
which involve six species ($A$, $B$, $C$, $D$, $E$ and $F$) and
three type of reactions.

The chemical reactions are
\begin{equation}\label{eq:chemical}
\begin{split}
A+B&\xrightarrow{k_1}C+F, \\
A+C&\xrightarrow{k_2}D+F, \\
A+D&\xrightarrow{k_3}E+F,
\end{split}
\end{equation}
and they are modeled with the differential equations
\begin{equation}\label{eq:reactions}
\begin{split}
\frac{\mathrm{d}A}{\mathrm{d}t} & = -k_1A\,B - k_2 A \,C -k_3A \, D, \\
\frac{\mathrm{d}B}{\mathrm{d}t} & = -k_1 A\, B, \\
\frac{\mathrm{d}C}{\mathrm{d}t} & = k_1 A\, B- k_2 A\,C,  \\
\frac{\mathrm{d}D}{\mathrm{d}t} & = k_2 A\,C - k_3 A\,D, \\
\frac{\mathrm{d}E}{\mathrm{d}t} & = k_3 A \,D.
\end{split}
\end{equation}
The task is to estimate the parameters $k_1$, $k_2$ and $k_3$ from experimental data
having the initial values given as $C(0) = D(0)=0$, $A(0) = 0.02090 \text{ mol/l}$ and $B(0) =A(0)/3$.
The data are shown in Figure~\ref{fig:himme}.

\begin{figure}[htb]
\centering
\includegraphics{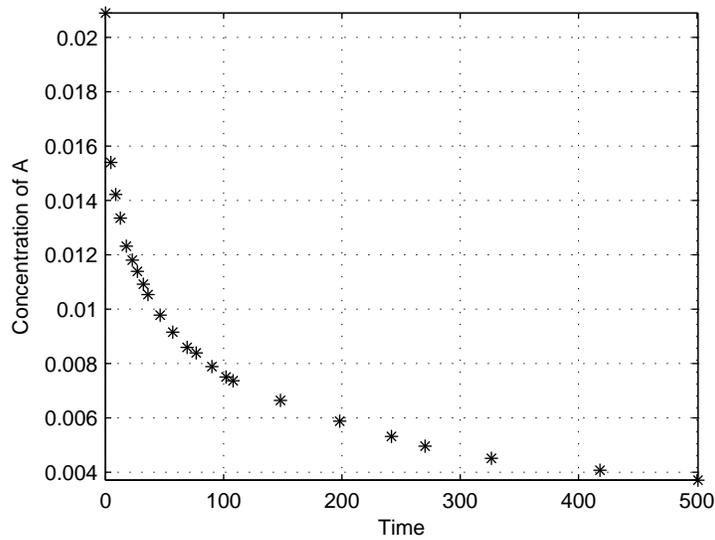}
\caption{Experimental data for concentration of $A$ in the chemical reaction model.}
\label{fig:himme}
\end{figure}

The reported parameter values were $k_1=14.7$, $k_2=1.53$ and $k_3=0.294$ \citep{Himmelblau:1970}.
We applied the VBAM algorithm to sample the parameters and then compared the results with the reported ones.
We used the random walk state space model \eqref{eq:RWM} with $\mathbf{A}_k = \mathbf{I}$,
$\mathbf{H}_k = \mathbf{I}$ and $\mathbf{Q}_k = 10^{-9}\mathbf{I}$.

\begin{figure}[htb]
\centering
\includegraphics{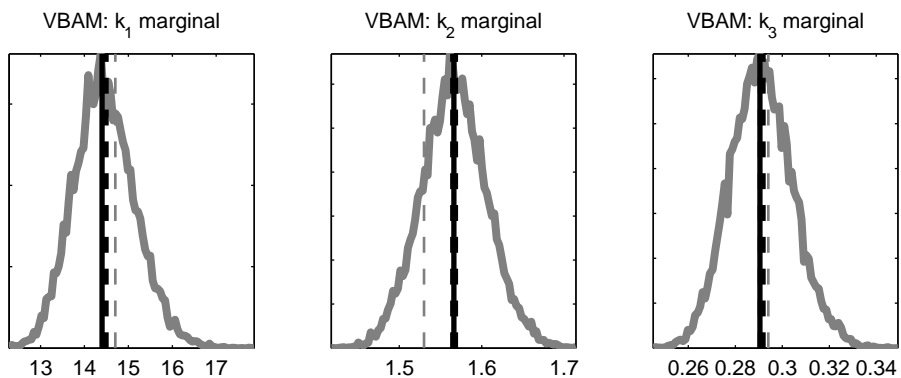}
\caption{Marginal distributions for parameters $k_1$, $k_2$, and $k_3$ in the chemical reaction model estimated from the samples
of VBAM plotted together with the sample means, maximum a posterior (MAP) estimates, and reported parameter values (RPV) in the chemical reaction model.
Here, the gray lines denote the distributions, the black dotted lines the sample means, the light dotted lines the RPV, and the black lines the MAP estimates. }
\label{fig:himm_marg}
\end{figure}

The marginal distribution plots estimated from 100,000 VBAM samples are shown in Figure \ref{fig:himm_marg}.
In the figure, it can be seen that the estimates are well consistent with the reported parameter values.
The scatter plots for VBAM samples are shown in Figure \ref{fig:himm_scatter}.
It is observed that there exist correlation between parameters, for example, $k_2$ and $k_3$
seems to have negative correlation. However, the correlation is not particularly strong and hence the parameters
are quite well identifiable from the data.

\begin{figure}[htb]
\centering
\includegraphics{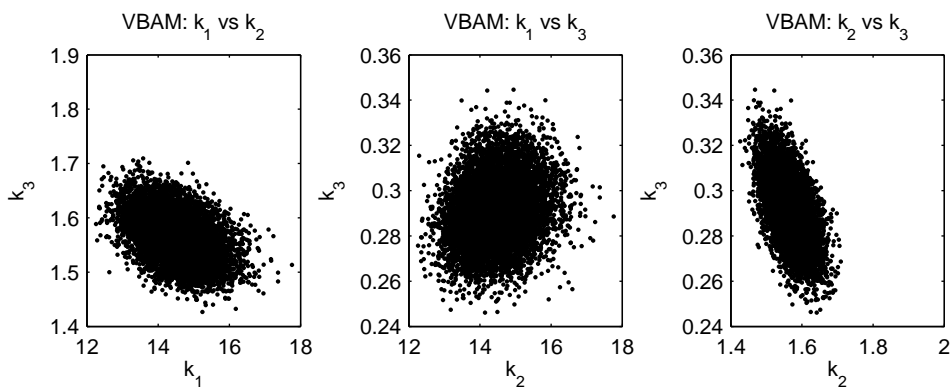}
\caption{VBAM scatter plots for the chemical reaction model. It is observed that there is some correlation between the parameters.}
\label{fig:himm_scatter}
\end{figure}
%
%

%
%%%
\subsection{Monod Model}

As the second real data example, we analyze the bacteria growth models studied in the book by \citet{Brown+Berthouex:2002}.
The estimation of the parameters of the Monod model has been studied for instance,
in Chapter 35 of the book \citep{Brown+Berthouex:2002}, where the authors used
the experimental data shown in Figure~\ref{fig:mono}.
The data were obtained by operating a continuous flow biological reactor at steady-state conditions
and the following Monod model was proposed to fit the data:
\begin{equation}\label{eq:mono}
y = \frac{\theta_1 \, x}{\theta_2 + x} + \epsilon.
\end{equation}
In Equation~\eqref{eq:mono}, $y$ is the growth rate expressed per hour
and is obtained at substrate concentration $x$, $\theta_1$ is the
maximum growth rate expressed per hour, $\theta_2$ is the
saturation constant, and $\epsilon$ is a Gaussian noise \citep{Brown+Berthouex:2002}.
In \citep{Brown+Berthouex:2002}, the parameters were estimated to be $\theta_1 = 0.153$ and $\theta_2 = 55.4$.

\begin{figure}[htb]
\centering
\includegraphics{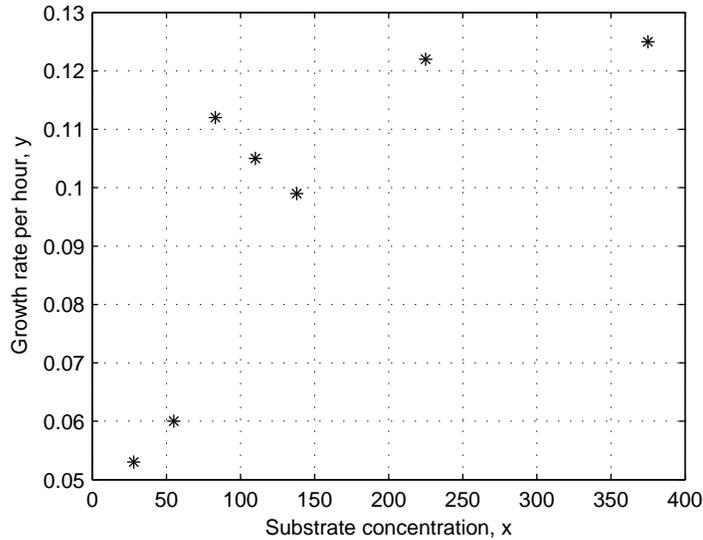}
\caption{Data for the bacteria growth (Monod) model.}
\label{fig:mono}
\end{figure}

We estimated these parameters using the VBAM algorithm. We used the random walk state space model \eqref{eq:RWM} for VBAM algorithm,  where
$\mathbf{A}_k = \mathbf{I}$, $\mathbf{H}_k = \mathbf{I}$ and $\mathbf{Q}_k = 10^{-9}\mathbf{I}$.

The marginal distribution estimates computed from 100,000 VBAM samples together with the sample means,
MAP estimates, and reported parameter values as well as the scatter plot are shown Figure~\ref{fig:mono_marg_scatter}.
The plots show that the reported parameter values are well within the estimated parameter distribution and that
there exist strong correlation between the parameters.

\begin{figure}[htb]
\centering
\includegraphics{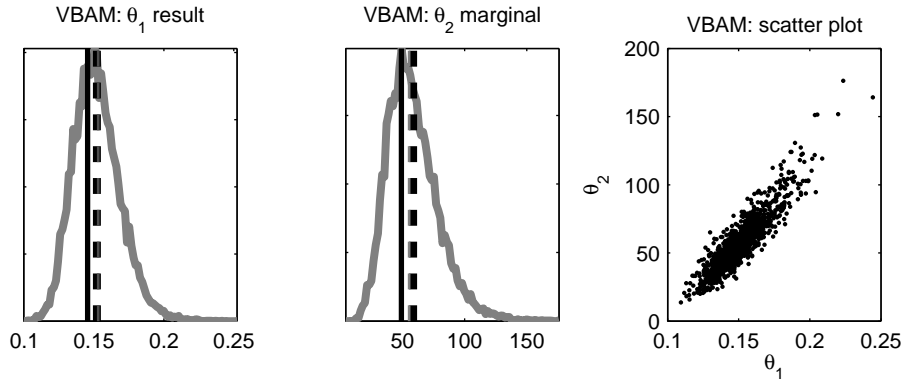}
\caption{The first two plots from left show the marginal distributions for
parameters $\theta_1$ and $\theta_2$ in the Monod model estimated from the samples
of VBAM, plotted along with the sample means, MAP, and reported parameter values (RPV).
The gray lines denote the distributions, the black dotted lines the sample means,
the light dotted lines the RPV, and the black lines the MAP estimates.
As can be seen, the parameter estimates are consistent with RPV. The plot on the right shows the scatter plot.
The parameters $\theta_1$ and $\theta_2$ seem to have a strong positive correlation.}
\label{fig:mono_marg_scatter}
\end{figure}

\section{Conclusion and Discussion} \label{sec:codi}
In this paper, we have proposed a new adaptive Markov chain Monte
Carlo (MCMC) method called \emph{variational Bayesian adaptive
  Metropolis (VBAM)} algorithm, which adapts the covariance matrix of
the Gaussian proposal distribution in the Metropolis algorithm with
the variational Bayesian adaptive Kalman filter (VB-AKF, 
\cite{Sarkka+Hartikainen:2013}). We have shown that the method
is indeed a valid adaptive MCMC method in the sense that it samples from the
correct target distribution by proving a strong law of large numbers
for it. We have numerically tested the performance of the method in
widely used example models and compared it to two other adaptive
MCMC schemes.  In the first two simulated experiments, our method
turned out to perform better than the AM algorithm of
\citet{Haario+Saksman+Tamminen:1999,Haario+Saksman+Tamminen:2001}. In
the third simulated example, the performance was similar to the performance of
the AM algorithm of \citet{Roberts+Rosenthal:2001}. In the two real data examples,
VBAM also produced results which are consistent with results reported in literature.

The advantage of the proposed method is that it has more parameters to
tune, which gives more freedom. In particular, the tight relationship with the linear systems
theory and Kalman filtering allows one to borrow good state space
models from target tracking literature
\citep{Jazwinski:1970,Shalom+Li+Kirubarajan:2001} and use them as the models in
VBAM. The correctness of the method can be easily verified by
checking that the resulting state space model is uniformly completely
observable and controllable, which is a standard step in building state
space models. Sometimes, however, the freedom of choosing algorithm parameters can
be seen as a disadvantage, because manual tuning of the VB-AKF model
parameters can turn out to be challenging. Fortunately, in many cases
a simple Gaussian random-walk state space model is a good default
choice.

The computational requirements of the VBAM method are typically $O(d^3)$,
where $d$ is the parameter dimensionality, while the complexity of a usual implementation of AM is $O(d^2)$. 
This is because the VB-AKF step is needed, which amounts to a (constant) number
of Kalman filter updates at each iteration and these operations are computationally
more demanding than what is needed in the standard AM. However, these operations
are still quite cheap and when the model is complex enough to require MCMC
sampling, the evaluation of the distribution can be expected to
dominate the computation time anyway. Furthermore, these operations can be typically
optimized for a given state space model. For example, in the random walk model we
do not actually need to perform all the matrix operations in full generality, because
the model matrices are diagonal. Even
though the basic implementation of the method is straightforward and not significantly harder than implementation
of an AM algorithm, developing an optimized version of the VBAM method for a particular type of state space model
can be more complicated.

An advantage of the method is that it can also easily be generalized in various ways.  For
example, we could extend it by replacing the linear Kalman filter with
non-linear Kalman filters such as the extended Kalman filter,
a sigma-point (unscented) filter, or even particle filters
\citep{Sarkka:2013}. In fact, provided that we can ensure that
the mean and covariance of the corresponding non-linear Kalman filter
remain bounded, replacing the linear VB-AKF with a non-linear one
\citep{Sarkka+Hartikainen:2013} should lead to a valid VBAM algorithm
as well. Similarly, a (Rao--Blackwellized) particle filter could be
used for estimating the noise covariance \citep{Sarkka:2013} and provided that the adaptation can be shown to
diminish in time. However, with non-linear state space model it will
be hard to find good state space models for the algorithm. This path
is interesting though, because it can lead to a completely new family
of adaptive MCMC algorithms, which utilize different kinds of filters
in the proposal adaptation.

\section*{Acknowledgments}
We are grateful to Arno Solin for proofreading this paper.
Isambi S.~Mbalawata was supported by the Finnish Centre of Excellence on
Inverse Problems Research of Academy of Finland.
Matti Vihola was supported by the Academy of Finland project 250575.
Simo S\"arkk\"a was supported by the Academy of Finland projects 266940 and 273475.

\bibliographystyle{model1b-num-names}
\bibliography{filter_amcmc}

%% not want to use model1b-num-names.bst.





% \end{thebibliography}

\end{document}